\documentclass[preprint,12pt]{elsarticle}

\usepackage{graphics}
\usepackage{graphicx}
\usepackage{amssymb}
\usepackage{amsthm}
\usepackage[fleqn,tbtags]{amsmath}
\usepackage{mathtools}
\usepackage{cases}
\usepackage{lipsum}
\usepackage{tikz}

\makeatletter
\def\ps@pprintTitle{%
 \let\@oddhead\@empty
 \let\@evenhead\@empty
 \def\@oddfoot{}%
 \let\@evenfoot\@oddfoot}
\makeatother

 \biboptions{}
\biboptions{sort&compress}
\newtheorem{thm}{Theorem}

\newtheorem{lem}[thm]{Lemma}
\theoremstyle{example}

\theoremstyle{definition}

\theoremstyle{remark}
\newtheorem{rem}{Remark}
\journal{}
\begin{document}
\begin{frontmatter}

\title{\Large Optimal critical exponent $L^{p}$ inequalities of Hardy type on the sphere via Xiao's method}

\author{Ahmed A. Abdelhakim}
\address{Mathematics Department, Faculty of Science, Assiut University, Assiut 71516, Egypt}
\ead{ahmed.abdelhakim@aun.edu.eg}

\begin{abstract}
First, we correct the proof presented in [Abimbola Abolarinwa, Kamilu Rauf, Songting Yin,
Sharp $L^{p}$ Hardy type and uncertainty
principle inequalities on the sphere, Journal of
Mathematical Inequalities,
13, 4 (2019), 1011 - 1022] and obtain a correct sharp version of an $L^{p}$ Hardy inequality on the sphere
$\mathbb{S}^{n}$ for all $2\leq p<n$. Secondly, we prove sharp critical exponent $L^{n}$ inequalities on the sphere $\mathbb{S}^{n}$ in $\mathbb{R}^{n+1}$, $n\geq 2$. The singularity in this problem is the geodesic distance from an arbitrary point on the sphere.
\end{abstract}


\end{frontmatter}
\section{Introduction}
\indent To our best knowledge, the first successful attempt to adapt the ideas introduced in \cite{Kombe}
to obtain inequalities of Hardy type on the $n$-dimensional sphere was that of Xiao's in
\cite{Xiao}. He obtained sharp $L^{2}$
Hardy inequalities on the Euclidean sphere $\mathbb{S}^{n}$, $n\geq 3$. Later, sharp critical case $L^{2}$ results were proved in \cite{ahmed} on the sphere $\mathbb{S}^{2}$ in $\mathbb{R}^{3}$.
Another extension was presented in \cite{xsun}
where subcritical optimal $L^{p}$ inequalities
were proved. The authors in \cite{ahmed,xsun,Xiao} considered the geodesic distance from the pole. In this case, the geodesic distance is precisely the angular variable. Very recently, the author in \cite{Songting}
improved the $L^{2}$ results in \cite{Xiao} by taking the singularity to be the geodesic distance from an arbitrary point on the sphere. This was followed by an attempt to obtain the corresponding sharp $L^{p}$ Hardy inequalities with the general geodesic distance in
\cite{Abimbola1}.
The author has proven in \cite{ahmed2} various sharp $L^{p}$ inequalities of Hardy type on the sphere in both the subcritical and critical exponent cases. The
way we prove sharpness of our inequalities in \cite{ahmed2} takes into account all the constants
involved. \\
\indent Let $u \in C^{\infty}(\mathbb{S}^{n})$, $n\geq 3$. Assume that $d$ denotes the geodesic distance on the sphere from an arbitrary point. It is claimed in (\cite{Abimbola1}, Theorem 1) that the following inequality holds for all $1<p<n$:
\begin{equation}\label{clmlp}
 \left(\frac{n-p}{p}\right)^{p} \int_{\mathbb{S}^{n}}
\frac{|u|^{p}d\sigma_{n}}{|\tan{d}|^{p}}
\leq   \int_{\mathbb{S}^{n}}
|\nabla f|^{p}d\sigma_{n}+
 \left(\frac{n-p}{p}\right)^{p-1}
\int_{\mathbb{S}^{n}}
\frac{|u|^{p}d\sigma_{n}}{\sin^{p-2}{d}}.
\end{equation}
The proof suggested in \cite{Abimbola1} is inaccurate. We point out a missing factor in that proof. Once corrected, the proof no longer implies the inequality (\ref{clmlp}). Interestingly, we obtain
an optimal $L^{p}$ inequality already proved in
\cite{ahmed2} using the divergence theorem, properties of the gradient and Laplacian of the geodesic distance, and H\"{o}lder and Young inequalities. Nevertheless, we believe that the inequality (\ref{clmlp}) probably
holds true.\\
\indent We also show how to adapt Xiao's method to obtain a sharp critical exponent $L^{n}$ Hardy type inequality on ${\mathbb{S}}^{n}$, $n\geq 2$, with the general geodesic distance from an arbitrary point on the sphere.
\section{Preliminaries}
Let $n\geq 2$ and let $\Theta_{n}:=(\theta_j)_{j=1}^{n}\in
[0,\pi]^{n-1}\times[0,2\pi]$.
We can assign to each point on the unit sphere
$\mathbb{S}^{n}$ in $\mathbb{R}^{n+1}$ the spherical coordinates parametrization $\left(x_{m}(\Theta_{n})\right)_{m=1}^{n+1}$, where
\begin{equation*}
x_{m}(\Theta_{n}):=\left\{
                \begin{array}{ll}
                  \cos{\theta_1}, & \hbox{$m=1$;} \vspace{0.1 cm} \\
                  \prod_{j=1}^{m-1}\sin{\theta_j}\cos{\theta_m}, & \hbox{$2\leq m\leq n$;} \vspace{0.15 cm}\\
                  \prod_{j=1}^{n}\sin{\theta_j}, & \hbox{$m=n+1$.}
                \end{array}
              \right.
\end{equation*}
With this representation, the surface gradient $\nabla_{\mathbb{S}^{n}}$ on the sphere $\mathbb{S}^{n}$  is defined by
\begin{equation*}
\nabla_{\mathbb{S}^{n}}  = \frac{\partial }{\partial \theta_1} \widehat{{\theta}_{1}} + \frac{1}{\sin \theta_1} \frac{\partial }{\partial \theta_2} \widehat{{\theta}_{2}}  + \cdots + \frac{1}{\sin \theta_1 \cdots \sin \theta_{n-1}} \frac{\partial }{\partial \theta_{n}} \widehat{{\theta}_{n}},
\end{equation*}
where $\left\{\widehat{{\theta}_{j}}\right\}$
is an orthonormal set of tangential vectors with the vector $\widehat{{\theta}_{j}}$ pointing in the direction of increase of ${\theta}_{j}$. In addition, the Laplace-Beltrami operator $\Delta_{\mathbb{S}^{n}}$ takes the form
\begin{align*}
\nonumber \hspace{-1 cm}\Delta_{\mathbb{S}^{n-1}}  =&\,
 \frac{1}{\sin^{n-2}{\theta_{1}}}
 \frac{\partial }{\partial \theta_1}
 \left(\sin^{n-2}{\theta_{1}} \frac{\partial }{\partial \theta_1} \right)+
\frac{1}{\sin^{2}{\theta_{1}}\sin^{n-3}{\theta_{2}}}
 \frac{\partial }{\partial \theta_2}
 \left(\sin^{n-3}{\theta_{2}} \frac{\partial }{\partial \theta_2} \right)+\\&\nonumber
+\dots+\frac{1}{\sin^{2}{\theta_{1}}\sin^{2}{\theta_{2}}...
 \sin^{2}{\theta_{n-2}}\sin{\theta_{n-1}}}
 \frac{\partial }{\partial \theta_{n-1}}
 \left(\sin{\theta_{n-1}} \frac{\partial }{\partial \theta_{n-1}} \right)+\\&+
  \frac{1}{\sin^{2}{\theta_{1}}\sin^{2}{\theta_{2}}...
 \sin^{2}{\theta_{n-1}}}
 \frac{\partial^2 }{\partial \theta^{2}_{n}}.
\end{align*}
Upon, identifying every point $(x_m(\Theta_{n-1}))_{m=1}^{n+1}\in \mathbb{S}^{n}$ with its parameters $\Theta_{n}$, the geodesic distance $d(\Theta_{n},\Phi_{n})$ from a point $\Phi_{n}
\in \mathbb{S}^{n}$ is defined by
\begin{equation}\label{gdsc}
d(\Theta_{n},\Phi_{n})=\arccos{\left(\sum_{m=1}^{n+1}
x_m(\Theta_{n})x_m(\Phi_{n})\right)}.
\end{equation}
The following properties of the geodesic distance on the sphere are proved in \cite{ahmed2}:
\begin{lem}\label{gradlap}
Suppose $\Phi_{n}$ is a point on the sphere
$\mathbb{S}^{n}$. Assume that $\,d(.,\Phi_{n-1}):
\mathbb{S}^{n}\rightarrow [0,\pi]\,$ is the geodesic distance from $\Phi_{n}$ on $\mathbb{S}^{n}$
defined in (\ref{gdsc}). Then
\begin{eqnarray}
\label{nablagdsc}
\left|\nabla_{\mathbb{S}^{n-1}}d\right|&=&
1,\\
\label{deltagdsc}\Delta_{\mathbb{S}^{n-1}}d
&=&(n-1)
\frac{\cos{d}}{\sin{d}}.
\end{eqnarray}
\end{lem}
We will also need the following basic inequality that can be found in \cite{lindqref}:
\begin{equation}\label{lindq}
|x+y|^{p}\geq |x|^{p}+
p|x|^{p-2}\langle x, y\rangle,\quad x,y\in \mathbb{R}^{n},\;p>1.
\end{equation}
\section{A correction of the proof in \cite{Abimbola1}}
Let $\Phi_{n}\in \mathbb{S}^{n}$, $n\geq 3$, and let $2\leq p<n$.
Let $u\in C^{\infty}(\mathbb{S}^{n})$ and write
\begin{equation}\label{dfn1}
u(\Theta_{n})=
\phi^{\alpha}{(\Theta_{n})}\,{\psi{(\Theta_{n})}},
\end{equation}
where $\phi(\Theta_{n}):=\sin{d(\Theta_{n},\Phi_{n})}$,
$\alpha=-{(n-p)}/{p}$. Clearly $\psi \in C^{\infty}\left(\mathbb{S}^{n}\right)$.
Since the geodesic metric
$d(\Theta_{n},\Phi_{n})=0$ only if $\Theta_{n}=\Phi_{n}$ and $d(\Theta_{n},\Phi_{n})=\pi$ only if
$\Theta_{n},\Phi_{n}$ are antipodal, then
$1/\phi \in C^{\infty}
\left(\mathbb{S}^{n}\setminus\left\{\pm \Phi_{n}\right\}\right)$. Taking the surface gradient of both sides of (\ref{dfn1}), then employing the inequality (\ref{lindq}), it follows that
on $\mathbb{S}^{n}\setminus\left\{\pm \Phi_{n}\right\}$ we have
\begin{align}
\nonumber
\hspace*{-1 cm}\left|\nabla_{\mathbb{S}^{n}} u\right|^{p}&=
\left|\alpha \phi^{\alpha-1}\psi
\nabla_{\mathbb{S}^{n}} \phi+\phi^{\alpha}
\nabla_{\mathbb{S}^{n}} \psi\right|^{p}
\\
&\nonumber \geq |\alpha|^{p}|\phi|^{\alpha p-p}|\psi|^{p}|\nabla_{\mathbb{S}^{n}} \phi|^{p}+\\
\nonumber &\quad+p|\alpha|^{p-2}|\phi|^{(\alpha-1)(p-2)}
|\psi|^{p-2}|\nabla_{\mathbb{S}^{n}} \phi|^{p-2}
\langle \alpha \phi^{\alpha-1}\psi
\nabla_{\mathbb{S}^{n}} \phi,\phi^{\alpha}
\nabla_{\mathbb{S}^{n}} \psi \rangle\\
&\nonumber = |\alpha|^{p}\phi^{\alpha p-p}|\psi|^{p}|\nabla_{\mathbb{S}^{n}} \phi|^{p}+\\
\label{gradf} &\quad
+\alpha|\alpha|^{p-2}
\phi^{\alpha p-p+1}
\left(p|\psi|^{p-2}\psi\right)
|\nabla_{\mathbb{S}^{n}} \phi|^{p-2}
\langle
\nabla_{\mathbb{S}^{n}} \phi,
\nabla_{\mathbb{S}^{n}} \psi \rangle,
\end{align}
since $\phi>0$. So far, our proof is in accordance with that in \cite{Abimbola1}. Since $p>1$, then $|\psi|^{p}$ is differentiable and we have
$\nabla_{\mathbb{S}^{n}}|\psi|^{p}=
p|\psi|^{p-2}\psi \nabla_{\mathbb{S}^{n}}\psi$.
Also, since $1/\phi$ is smooth on
on $\mathbb{S}^{n}\setminus\left\{\pm \Phi_{n}\right\}$ and $\alpha p-p+2=-(n-2)\neq 0$, then we can write $\phi^{\alpha p-p+1}\nabla_{\mathbb{S}^{n}}\phi=
\frac{1}{\alpha p-p+2}\nabla_{\mathbb{S}^{n}}\phi^{\alpha p-p+2}$.
Using this in (\ref{gradf}) implies
\begin{align}
\nonumber
\left|\nabla_{\mathbb{S}^{n}} u\right|^{p}\geq& |\alpha|^{p}\phi^{\alpha p-p}|\psi|^{p}|\nabla_{\mathbb{S}^{n}} \phi|^{p}+\\
&\label{gradf1}+\frac{\alpha|\alpha|^{p-2}}{\alpha p-p+2}|\nabla_{\mathbb{S}^{n}} \phi|^{p-2}
\langle\nabla_{\mathbb{S}^{n}} \phi^{\alpha p-p+2},
\nabla_{\mathbb{S}^{n}} |\psi|^{p} \rangle.
\end{align}
At this point, the factor $|\nabla_{\mathbb{S}^{n}} \phi|^{p-2}$ went unjustifiably missing in \cite{Abimbola1}. Moreover, the vector
$\nabla_{\mathbb{S}^{n}} |\psi|^{p}$ is confused with $\nabla_{\mathbb{S}^{n}} \psi^{p}$.
The next main step in \cite{Abimbola1}
is to write
\begin{equation*}
\langle
\nabla_{\mathbb{S}^{n}} \phi^{\alpha p-p+2},
\nabla_{\mathbb{S}^{n}} |\psi|^{p} \rangle=
\text{div}\left(|\psi|^{p}\nabla_{\mathbb{S}^{n}} \phi^{\alpha p-p+2}\right)-|\psi|^{p}
\Delta_{\mathbb{S}^{n}} \phi^{\alpha p-p+2},
\end{equation*}
then use the divergence theorem that yields
$\int_{\mathbb{S}^{n}}\text{div}\left(|\psi|^{p}
\nabla_{\mathbb{S}^{n}} \phi^{\alpha p-p+2}\right)d\sigma_{n}=0$.
If the divergence theorem is to be used, it should rather be applied
to $|\nabla_{\mathbb{S}^{n}} \phi|^{p-2}$ $\text{div}\left(|\psi|^{p}\nabla_{\mathbb{S}^{n}} \phi^{\alpha p-p+2}\right)$ whose integral does not simply vanish.\\
\indent Let us proceed from (\ref{gradf1}). Substituting for
$\alpha$, $\phi$ and $\psi$, then using (\ref{nablagdsc}), and integrating
both sides over $\mathbb{S}^{n}$, we find
\begin{equation}\label{gradf11}
\int_{\mathbb{S}^{n}}\left|\nabla_{\mathbb{S}^{n}} u\right|^{p}d\sigma_{n}\geq \int_{\mathbb{S}^{n}}\left(\frac{n-p}{p}\right)^{p}
\frac{|u|^{p}}{|\tan{d}|^{p}}d\sigma_{n}+
\frac{1}{n-2}\left(\frac{n-p}{p}\right)^{p-1}
I_{n,p},
\end{equation}
where
\begin{equation*}
I_{n,p}=\int_{\mathbb{S}^{n}}|\cos{d}|^{p-2}
\left\langle
\nabla_{\mathbb{S}^{n}} \frac{1}{\sin^{n-2}{d}},
\nabla_{\mathbb{S}^{n}}\left( |u|^{p}\sin^{n-p}{d} \right)\right\rangle
d\sigma_{n}.
\end{equation*}
Now, observe that, when $p\geq2$, we can make sense of the gradient
\begin{equation*}
\nabla_{\mathbb{S}^{n}}|\cos{d}|^{p-2}\cos{d}=
-(p-1)|\cos{d}|^{p-2}\sin{d}\,
\nabla_{\mathbb{S}^{n}} d.
\end{equation*}
Therefore, using (\ref{deltagdsc}), we may simplify $I_{n,p}$ by integration by parts on the compact manifold $\mathbb{S}^{n}$ to obtain
\begin{eqnarray*}
I_{n,p}&=&(n-2)\int_{\mathbb{S}^{n}}
|u|^{p}\sin^{n-p}{d}\:
\text{div}\left(|\cos{d}|^{p-2}\cos{d}
\frac{\nabla_{\mathbb{S}^{n}}\,d}{\sin^{n-1}{d}}\right)
d\sigma_{n}\\
&=&-(n-2)(p-1)\int_{\mathbb{S}^{n}}
|u|^{p}\sin^{n-p}{d}\frac{|\cos{d}|^{p-2}}
{\sin^{n-2}{d}}
d\sigma_{n}\\
&=&-(n-2)(p-1)\int_{\mathbb{S}^{n}}
\frac{|u|^{p}}{|\tan{d}|^{p-2}}
d\sigma_{n},
\end{eqnarray*}
because, using (\ref{nablagdsc}) and (\ref{deltagdsc}), it turns out that
\begin{equation}\label{bcsnbd}
\left\langle\nabla_{\mathbb{S}^{n}}
\frac{1}{\sin^{n-1}{d}},\nabla_{\mathbb{S}^{n}} d \right\rangle
=-\frac{\Delta_{\mathbb{S}^{n}} d}{\sin^{n-1}{d}},\;
\Theta_{n}\neq\pm \Phi_{n}.
\end{equation}
Plugging the integral $I_{n,p}$ into the inequality (\ref{gradf11}) we deduce the inequality
\begin{align}
\nonumber
\int_{\mathbb{S}^{n}}\left|\nabla_{\mathbb{S}^{n}} u\right|^{p}d\sigma_{n}+&
(p-1)\left(\frac{n-p}{p}\right)^{p-1}\int_{\mathbb{S}^{n}}
\frac{|u|^{p}}{|\tan{d}|^{p-2}}
d\sigma_{n}
\geq \\ &\label{gradf2}\left(\frac{n-p}{p}\right)^{p}\int_{\mathbb{S}^{n}}
\frac{|u|^{p}}{|\tan{d}|^{p}}d\sigma_{n}.
\end{align}
\begin{rem}
The inequality (\ref{gradf2}) is obtained using a different method in \cite{ahmed2}. It is also shown in \cite{ahmed2} that all three coefficients in (\ref{gradf2}) are optimal  for all
$2\leq p<n$, using optimizing sequences
in the Sobolev space $W^{1,p}(\mathbb{S}^{n})$. By density, such functions can be approximated by smooth functions to give optimizing sequences in $C^{\infty}(\mathbb{S}^{n})$.
\end{rem}
\begin{rem}
When $p=2$, the inequality (\ref{gradf2}) reduces to the main inequality proved in (\cite{Songting}, Theorem 1.1).
\end{rem}
\section{Critical $L^{n}(\mathbb{S}^{n})$ Hardy inequality}
We show how to apply Xiao's method \cite{Xiao} to prove an optimal critical exponent Hardy inequality on $\mathbb{S}^{n}$, $n\geq 2$, considering the geodesic distance
$d(.,\Phi_{n})$ defined in (\ref{gdsc}) from an arbitrary point $\Phi_{n}\in \mathbb{S}^{n}$. We will certainly make use of the properties (\ref{nablagdsc})
and (\ref{deltagdsc}).
\begin{thm}
Fix $n\geq 2$ and assume $u \in C^{\infty}(\mathbb{S}^{n})$. Let $\Phi_{n}$
be some point on the sphere $\mathbb{S}^{n}$ and consider the geodesic distance (\ref{gdsc}) from $\Phi_{n}$. Then
\begin{align}
\nonumber
\int_{\mathbb{S}^{n}}\left|\nabla_{\mathbb{S}^{n}} u\right|^{n}d\sigma_{n}+&
(n-1)\left(\frac{n-1}{n}\right)^{n-1}\int_{\mathbb{S}^{n}}
\frac{|u|^{p}}{|\tan{d}|^{n-2}
\left(\log{\frac{e}{\sin{d}}}\right)^{n-1}}
d\sigma_{n}
\geq \\ &\label{gradf3}\left(\frac{n-1}{n}\right)^{n}
\int_{\mathbb{S}^{n}}
\frac{|u|^{n}}{|\tan{d}|^{n}
\left(\log{\frac{e}{\sin{d}}}\right)^{n}}d\sigma_{n}.
\end{align}
\end{thm}
\begin{proof}
Suppose $u\in C^{\infty}(\mathbb{S}^{n})$.
Analogously to (\ref{dfn1}), we can write
\begin{equation*}
u(\Theta_{n})=
\phi^{\alpha}{(\Theta_{n})}\,{\psi{(\Theta_{n})}},
\end{equation*}
where $\psi \in C^{\infty}(\mathbb{S}^{n})$, but we modify the definition of $\alpha$ and that of $\phi$ as follows:
\begin{equation*}
\phi(\Theta_{n}):=
\log{\frac{e}{\sin{d(\Theta_{n},\Phi_{n})}}},\quad
\alpha=\frac{n-1}{n}.
\end{equation*}
We note here that $\phi\in C^{1}\left(\mathbb{S}^{n}\setminus\left\{\pm \Phi_{n}\right\}\right)$ and we have
\begin{equation}\label{grdlogsn}
\nabla_{\mathbb{S}^{n}}
\phi(\Theta_{n})=-\frac{1}{\tan{d(\Theta_{n},\Phi_{n})}}
\nabla_{\mathbb{S}^{n}}{d(\Theta_{n},\Phi_{n})},\quad
\Theta_{n}\neq\pm \Phi_{n}.
\end{equation}
Consequently, in the light of (\ref{nablagdsc}), we see that
\begin{equation}\label{grdlogsn12}
\left|\nabla_{\mathbb{S}^{n}}
\phi(\Theta_{n})\right|=\frac{1}{
\left|\tan{d(\Theta_{n},\Phi_{n})}\right|}.
\end{equation}
We pick up the proof at (\ref{gradf1}) with $p=n$.
Substituting for $\alpha$, $\phi$ and $\psi$, while using (\ref{grdlogsn}) and (\ref{grdlogsn12}), we get
\begin{equation}\label{pc1}
|\alpha|^{n}\phi^{\alpha n-n}|\psi|^{n}|\nabla_{\mathbb{S}^{n}} \phi|^{n}
=\left(\frac{n-1}{n}\right)^{n}
\frac{|u|^{n}}{\left|\tan{d}\right|^{n}
\left(\log{\frac{e}{\sin{d}}}\right)^{n}},\quad
\Theta_{n}\neq \pm \Phi_{n},
\end{equation}
\begin{equation}\label{pc2}
\begin{split}
&\frac{\alpha|\alpha|^{n-2}}{\alpha n-n+2}|\nabla_{\mathbb{S}^{n}} \phi|^{n-2}
\langle\nabla_{\mathbb{S}^{n}} \phi^{\alpha n-n+2},
\nabla_{\mathbb{S}^{n}} |\psi|^{n} \rangle\\
&=
\left(\frac{n-1}{n}\right)^{n-1}
\frac{1}{\left|\tan{d}\right|^{n-2}}
\left\langle\nabla_{\mathbb{S}^{n}} \left(\log{\frac{e}{\sin{d}}}\right),
\nabla_{\mathbb{S}^{n}} |\psi|^{n} \right\rangle\\
&=-
\left(\frac{n-1}{n}\right)^{n-1}
\frac{1}{\left|\tan{d}\right|^{n-2}\tan{d}}
\left\langle\nabla_{\mathbb{S}^{n}} d,
\nabla_{\mathbb{S}^{n}} |\psi|^{n} \right\rangle,\quad
\Theta_{n}\neq \pm \Phi_{n}.
\end{split}
\end{equation}
Taking into account the calculations (\ref{pc1}) and (\ref{pc2}), we integrate both sides of (\ref{gradf1}),  with $p=n$, over $\mathbb{S}^{n}$. It follows that
\begin{equation}\label{gradf301}
\begin{split}
&\int_{\mathbb{S}^{n}}\left|\nabla_{\mathbb{S}^{n}} u\right|^{n}d\sigma_{n}\geq
\left(\frac{n-1}{n}\right)^{n}\int_{\mathbb{S}^{n}}
\frac{|u|^{n}}{|\tan{d}|^{n}
\left(\log{\frac{e}{\sin{d}}}\right)^{n}}d\sigma_{n}-
\left(\frac{n-1}{n}\right)^{n-1}J_{n,p},
\end{split}
\end{equation}
where
\begin{equation*}
J_{n,p}=\int_{\mathbb{S}^{n}}
\frac{\left|\cos{d}\right|^{n-2}\cos{d}}{
\sin^{n-1}{d}}
\left\langle\nabla_{\mathbb{S}^{n}} d,
\nabla_{\mathbb{S}^{n}} |\psi|^{n} \right\rangle
d\sigma_{n}
\end{equation*}
Invoking the divergence theorem, we see that
\begin{equation*}
J_{n,p}=-\int_{\mathbb{S}^{n}}
|\psi|^{n}\text{div}
\left(\frac{\left|\cos{d}\right|^{n-2}\cos{d}}{
\sin^{n-1}{d}}
\nabla_{\mathbb{S}^{n}} d  \right)
d\sigma_{n}
\end{equation*}
Using the identity (\ref{bcsnbd}), we immediately realize
\begin{eqnarray}
\nonumber J_{n,p}&=&-\int_{\mathbb{S}^{n}}
\frac{|u|^{n}}{\left(\log{\frac{e}{
\sin{d(\Theta_{n},\Phi_{n})}}}\right)^{n-1}}
\left\langle \frac{\nabla_{\mathbb{S}^{n}} d}{\sin^{n-1}{d}},
\nabla_{\mathbb{S}^{n}}
{\left|\cos{d}\right|^{n-2}\cos{d}} \right\rangle
d\sigma_{n}\\
&=&\label{fnl}
(n-1)\int_{\mathbb{S}^{n}}
\frac{|u|^{n}}{\left(\log{\frac{e}{
\sin{d(\Theta_{n},\Phi_{n})}}}\right)^{n-1}}
\frac{1}{\left|\tan{d}\right|^{n-2}}
d\sigma_{n}.
\end{eqnarray}
The inequality (\ref{gradf3}) finally
follows from (\ref{gradf301}) and (\ref{fnl}).
\end{proof}
\begin{rem}
The inequality (\ref{gradf3}) is derived in \cite{ahmed2} using a different method.
It is noteworthy that all constants of (\ref{gradf3})
are optimal. This is also proved in \cite{ahmed2}
utilizing optimizing sequences in the Sobolev space
$W^{1,n}(\mathbb{S}^{n})$. The constants are therefore optimal for smooth functions. For, arguing by contradiction, if a constant in (\ref{gradf3}) could be improved for smooth functions, then the improved inequality would also be valid for
$W^{1,n}(\mathbb{S}^{n})$ functions by density.
\end{rem}

\end{document}